\documentclass[11pt]{article}
\usepackage[english,activeacute]{babel}
\usepackage{amsmath,amsfonts,amssymb,amstext,amsthm,amscd,mathrsfs,amsbsy,centernot}
\usepackage{xcolor}
\usepackage{hyperref}

\newtheorem{thm}{Theorem}[section]
\newtheorem{prop}[thm]{Proposition}

\newtheorem{lem}[thm]{Lemma}

\theoremstyle{definition}
\newtheorem{defn}[thm]{Definition}
\newtheorem{exam}[thm]{Example}
\newtheorem{rem}[thm]{Remark}

\newtheorem{theoremx}{Theorem}

\newcommand{\N}{\mathbb N}
\newcommand{\Z}{\mathbb Z}

\newcommand{\R}{\mathbb R}
\newcommand{\C}{\mathbb C}

\newcommand{\V}{\mathbf{V}}

\newcommand{\sing}{\operatorname{Sing}}
\newcommand{\gr}{\operatorname{Gr(d,n)}}

\newcommand{\sd}{C}

\newcommand{\conv}{\operatorname{Conv}}

\newcommand{\Ga}{\Gamma}
\newcommand{\ga}{\gamma}

\newcommand{\xg}{X_{\Gamma}}
\newcommand{\ig}{I_{\Gamma}}

\newcommand{\Nw}{\mathcal N(\Ga)} 
\newcommand{\vn}{V(\mathcal N(\Ga))} 

\providecommand{\keywords}[1]{{\textbf{Keywords:}} #1}
\providecommand{\msc}[1]{{\textbf{MSC:}} #1}

\begin{document}

\title{Nash blowups of normal toric surfaces: the case of one and two segments}

\author{Daniel Duarte \footnote{Research supported by CONAHCYT project CF-2023-G33 and PAPIIT grant IN117523.}, Jawad Snoussi \footnote{Research supported by PAPIIT grant IN117523.}}

\maketitle

\begin{abstract}
We show that iterating Nash blowups resolves the singularities of normal toric surfaces satisfying the following property: the minimal generating set of the corresponding semigroup is contained in one or two segments. We also provide examples with an arbitrary number of segments for which the same result holds.
\end{abstract}

\noindent\keywords{Nash blowups, normal toric surface, resolution of singularities, compact edges.}
\\
\msc{14B05,14E15,14M25.}

\section*{Introduction} 

The Nash blowup of an algebraic variety is a modification that replaces singular points by the tangency at non-singular nearby points. It has been proposed to resolve singularities by iterating Nash blowups or Nash blowups followed by normalization \cite{S,No,GS1}. A great deal of work has been done on these questions \cite{No,R,GS1,GS2,Hi,Sp,GS3,EGZ,GT,GM,D1,At,DG,DN1,DJNB,DDS,CDLL}. In a striking turn of events, the answer turned out to be negative in dimensions four or higher \cite{CDLL}. 

On the other hand, a celebrated theorem of M. Spivakovsky states that iterating normalized Nash blowups gives a resolution of complex surfaces \cite{Sp}. However, very little is known regarding the resolution properties of Nash blowups without normalizing in dimension two. 

In his Ph. D. thesis, V. Rebassoo gave an affirmative answer to Nash's question for the surfaces $\{x^ay^b-z^c=0\}\subset\C^3$, for $a,b,c\in\N$ having no common divisor \cite{R}. In other words, the result holds for toric surfaces whose defining semigroup can be minimally generated by three elements. In this paper we provide new families of toric surfaces for which Nash blowups give a resolution in characteristic zero. Our first main theorem is the following. 

\begin{theoremx}\label{thm intro 1}
Let $\sd\subset\R^2$ be a strongly convex rational polyhedral cone of dimension $2$ and $\Ga$ be the minimal generating set of the semigroup $\sd\cap\Z^2$. Let $\Theta$ be the convex hull of $\sd\cap\Z^2\setminus\{(0,0)\}$. Assume that $\Theta$ has only one or two compact edges. Then the iteration of Nash blowups resolves the singularities of the toric variety defined by $\Ga$.
\end{theoremx}

We want to emphasize that this is the first result since Rebassoo's theorem in 1977 that provides new families of toric surfaces that can be resolved via Nash blowups (without normalizing). In this context, only local uniformization for some valuations were obtained so far \cite{GT,D1}. 

Our method of proof consists essentially in reducing the problem to the situation of Rebassoo's theorem. Firstly, we use the combinatorial description of Nash blowups of toric varieties \cite{GT}. We show that by applying once the combinatorial algorithm to the family described in the theorem, there are three possible outcomes for the semigroup defining each affine chart of the Nash blowup: it is generated by two elements, it is generated by three elements, or it appears as a semigroup obtained after applying once the algorithm to some other semigroup generated by three elements. In the first case we have a non-singular toric surface. The other two cases follow by Rebassoo's theorem.

It is natural to ask whether our techniques apply to saturated semigroups contained in more than two compact edges. We provide an example of a saturated semigroup whose minimal generating set is contained in three compact edges but for which our techniques cannot be applied directly to obtain a resolution. Nevertheless, as an invitation to further investigate the resolution properties of Nash blowups of normal toric surfaces in characteristic zero, we provide a family having an arbitrary number of compact edges that can be resolved via Nash blowups.

\begin{theoremx}\label{thm intro 2}
Consider the Fibonacci sequence: $f_1:=1$, $f_2:=1$, and $f_j:=f_{j-1}+f_{j-2}$, for $j>2$. Let $l\geq 4$ be an even integer. Denote 
$$\Ga=\{(1,0),(f_1,f_2),(f_3,f_4),\ldots,(f_{l-1},f_{l})\}\subset\Z^2.$$
Then two iterations of Nash blowup resolve the singularities of the toric surface defined by $\Ga$.
\end{theoremx}

The paper is organized as follows. In the first section, we recall the basic definitions we use and describe the combinatorial algorithm that corresponds to the Nash blowup of a toric surface. Section 2 is devoted to proving Theorem \ref{thm intro 1}. This proof is divided in several cases, depending on the configuration of the minimal generating set. In section 3, we prove Theorem \ref{thm intro 2}. Finally, section 4 contains a discussion on the continued fractions associated to the normal toric surfaces of Theorems \ref{thm intro 1} and \ref{thm intro 2}.

\subsection*{Acknowledgments}
We would like to thank the referee for the careful reading and valuable remarks on the paper. These comments were very helpful in improving the presentation and clarifying several proofs.


\section{Nash blowup and toric varieties} 

Let us start by recalling the definition of the Nash blowup of an algebraic variety. 

\begin{defn}
Let $X\subset\C^n$ be an irreducible algebraic variety of dimension $d$. Consider the Gauss map:
\begin{align}
G:X\setminus\sing(X)&\rightarrow \gr\notag\\
x&\mapsto T_xX,\notag
\end{align}
where $\gr$ is the Grassmannian of $d$-dimensional vector spaces in $\C^n$, and $T_xX$ is the tangent space to $X$ at $x$. Denote by $X^*$ the Zariski closure of the graph of $G$. Call $\nu$ the restriction to $X^*$ of the projection of $X\times\gr$ to $X$. The pair $(X^*,\nu)$ is called the Nash blowup of $X$.
\end{defn}

In this paper we investigate the resolution of singularities properties of Nash blowups for certain families of toric surfaces. Let us recall the definition of an affine toric variety (see, for instance, \cite[Section 1.1]{CLS} or \cite[Chapters 4, 13]{St}).

Let $\Gamma=\{\gamma_1,\ldots,\gamma_n\}\subset\Z^d$. 
The set $\Gamma$ induces a homomorphism of semigroups
\begin{align}
\pi_{\Ga}:\N^n\rightarrow\Z^d,\mbox{ }\mbox{ }\mbox{ }\alpha=(\alpha_1,\ldots,\alpha_n)\mapsto \alpha_1\gamma_1+\cdots+\alpha_n\gamma_n. \notag
\end{align}
Consider the ideal
$$\ig:=\langle x^{\alpha}-x^{\beta}|\alpha,\beta\in\N^n,\mbox{ }\pi_{\Ga}(\alpha)=\pi_{\Ga}(\beta)\rangle\subset\C[x_1,\ldots,x_n].$$
We denote $\xg:=\V(\ig)\subset\C^n$ and call it a toric variety defined by $\Gamma$. 

Throughout this paper we use the following notation. Given $\Ga\subset\Z^d$, we denote by $\N\Ga$ the semigroup generated by $\Ga$, by $\Z\Ga$ the group generated by $\Ga$, and by $\R_{\geq0}\Ga$ the cone generated by $\Ga$. We say that the cone $\R_{\geq0}\Ga$ is strictly convex if the only vector subspace contained in $\R_{\geq0}\Ga$ is $\{0\}$. Given a subset $S\subset\R^d$, we denote by $\conv(S)$ its convex hull. Moreover, a hyperplane given by $\{L=0\}$, for a linear map $L:\R^d\to \R$, is said to be a supporting hyperplane of $S$ if $S\subset\{L\geq0\}$.

The following theorem is one of the main ingredients for the results presented here.

\begin{thm}\label{Reb}\cite[Theorem 3.1]{R}
Let $X=\{x^ay^b-z^c=0\}\subset\C^3$, $a,b,c\in\N$ having no common divisor. Then the iteration of Nash blowups resolves the singularities of $X$. In other words, this result applies to toric surfaces $\xg$ defined by $\Ga=\{\gamma_1,\gamma_2,\gamma_3\}\subset\Z^2$ such that $\Z\Ga=\Z^2$ and $\R_{\geq0}\Ga$  strictly convex.
\end{thm}


\subsection{Combinatorial description of Nash blowups of toric surfaces} 

The Nash blowup of a toric variety can be described combinatorially in terms of the semigroup of the variety. We are interested in  the case of toric surfaces so we present the description in this particular case (see \cite[Section 3.9.2]{Sp2}).

Let $\Ga=\{\ga_1,\cdots,\ga_n\}\subset\Z^2$ be a subset such that $\Z\Ga=\Z^2$ and the cone $\sd=\R_{\geq0}\Ga$ is strictly convex. Let $\Ga_0=\{\ga_i+\ga_j|\det(\gamma_i\mbox{ }\gamma_j)\neq0\}$. Let $\{\ga_{i_0},\ga_{j_0}\}\subset\Ga$ be such that $\ga_{i_0}+\ga_{j_0}\in\Ga_0$. Consider
\begin{align}\label{A}
A(\ga_{i_0})&=\{\gamma_k-\gamma_{i_0}|k\in\{1,\ldots,n\}\setminus\{i_0,j_0\},\mbox{ }\det(\gamma_k\mbox{ }\gamma_{j_0})\neq0\},\notag\\
A(\ga_{j_0})&=\{\gamma_k-\gamma_{j_0}|k\in\{1,\ldots,n\}\setminus\{i_0,j_0\},\mbox{ }\det(\gamma_k\mbox{ }\gamma_{i_0})\neq0\}.
\end{align}
Let $\Ga(\ga_{i_0},\ga_{j_0})=\{\gamma_{i_0},\gamma_{j_0}\}\cup A(\ga_{i_0})\cup A(\ga_{j_0})$. Then the Nash blowup of $\xg$ has an open affine covering given by the toric surfaces defined by $\Ga(\ga_{i},\ga_{j})$, where $\ga_i+\ga_j\in \Ga_0$.

This description can be refined as follows. Let
$$\Nw=\conv(\{(\ga_i+\ga_j)+\sd|\ga_i+\ga_j\in \Ga_0\}).$$
Then 
$$\{X_{\Ga(\ga_i,\ga_j)}|\ga_i+\ga_j\mbox{ is a vertex of }\Nw\}$$ 
is an open affine covering of the Nash blowup of $\xg$. In other words, we do not need to consider all $\ga_i+\ga_j\in \Ga_0$ but only those such that $\ga_i+\ga_j$ is a vertex of $\Nw$ (see \cite[Propositions 32 and 60]{GT}).

Now we consider a case where the vertices of $\Nw$ can be described explicitly. Let $\sd\subset\R^2$ be a strongly convex rational polyhedral cone of dimension $2$. It is known that the minimal generating set of $\sd\cap\Z^2$ can be described as follows. Let $\Theta\subset\R^2$ be the convex hull of $\sd\cap\Z^2\setminus\{(0,0)\}$. Let $\Ga\subset\Z^2$ be the points lying on the compact edges of $\Theta$. Then $\Ga$ is the minimal generating set of $\sd\cap\Z^2$ \cite[Proposition 1.21]{Oda}. 

\begin{prop}\label{vertices}
With the previous notation, let $\Ga=\{\ga_1,\ldots,\ga_n\}\subset\Z^2$ be the minimal generating set of $\sd\cap\Z^2$, ordered counterclockwise. Let $\{\ga_1=\ga_{i_1},\ga_{i_2},\ldots,\ga_{i_r}=\ga_n\}\subset \Ga$, $1=i_1<i_2<\cdots<i_r=n$, be the set of vertices of $\Theta$. Denote by $V(\Nw)$ the vertices of $\Nw$. Then
$$V(\Nw)=\{\ga_1+\ga_2,\ga_{n-1}+\ga_{n}\}\cup\{\ga_{i_j-1}+\ga_{i_j},\ga_{i_j}+\ga_{i_j+1}|1<j<r\}.$$
\end{prop}
\begin{proof}
Firstly, it is known that $V(\Nw)\subset\{\ga_i+\ga_{i+1}|1\leq i<n\}$ \cite[Propositions 4.11 and 4.12]{At}. Furthermore, if a compact edge of $\Theta$ contains four consecutive points, say, $\{\ga_{j},\ga_{j+1},\ga_{j+2},\ga_{j+3}\}$, then $\ga_{j+1}+\ga_{j+2}$ is in the relative interior of the segment $[\ga_j+\ga_{j+1},\ga_{j+2}+\ga_{j+3}]$. In particular, $\ga_{j+1}+\ga_{j+2}\notin V(\Nw)$. Hence, 
$$V(\Nw)\subset\{\ga_1+\ga_2,\ga_{n-1}+\ga_{n}\}\cup\{\ga_{i_j-1}+\ga_{i_j},\ga_{i_j}+\ga_{i_j+1}|1<j<r\}.$$

To show the other inclusion, consider first the case $1<j<r$. Let $l\subset\R^2$ be the line joining $\gamma_{i_j}$ and $\gamma_{i_{j+1}}$. Then, a translation of this line is a supporting hyperplane of $\Theta$ and also of $\Nw$. In addition, $l$ reaches its minimum value with respect to $\Nw$ on the segment $[\gamma_{i_j}+\ga_{i_j+1},\gamma_{i_{j+1}-1}+\gamma_{i_{j+1}}]$. By perturbing $l$ slightly, we deduce that $\gamma_{i_j}+\ga_{i_j+1}$ and $\gamma_{i_{j+1}-1}+\gamma_{i_{j+1}}$ are vertices of $\Nw$. A similar argument shows that $\gamma_1+\gamma_2$ and $\ga_{n-1}+\ga_{n}$ are vertices as well.
\end{proof}

\begin{rem}
As before, let $\Ga$ be the minimal generating set of $C\cap\Z^2$. Then any two different vectors in $\Ga$ are linearly independent. Hence, for a vertex $\ga_i+\ga_j$ in $\vn$, the sets defined in (\ref{A}) become
\begin{align}\label{A}
A(\ga_{i})&=\{\gamma-\gamma_{i}|\ga\in\Ga\setminus\{\ga_i,\ga_j\}\},\notag\\
A(\ga_{j})&=\{\gamma-\gamma_{j}|\ga\in\Ga\setminus\{\ga_i,\ga_j\}\}.\notag
\end{align}
We use this remark throughout this paper.
\end{rem}

The following example illustrates the technique we use to prove our main theorem.

\begin{exam}
Let $\Ga=\{\ga_1=(1,0),\ga_2=(1,1),\ga_3=(1,2),\ga_4=(3,7)\}$. This set is the minimal generating set of $\sd\cap\Z^2$, where $\sd=\R_{\geq0}((1,0),(3,7))$. We show that iterating Nash blowups resolves the singularities of $\xg$. Firstly, by Proposition \ref{vertices}, 
$$V(\Nw)=\{\ga_1+\ga_2,\ga_2+\ga_3,\ga_3+\ga_4\}.$$
\begin{itemize}
\item Consider $\ga_1+\ga_2\in V(\Nw)$. Then 
\end{itemize}
\begin{align}
A(\ga_1)&=\{(0,2),(2,7)\},\notag\\
A(\ga_2)&=\{(0,1),(2,6)\}.\notag 
\end{align}
Hence, $\Ga(\ga_1,\ga_2)=\{(1,0),(1,1)\}\cup\{(0,2),(2,7),(0,1),(2,6)\}$. The minimal generating set of $\N\Ga(\ga_1,\ga_2)$ is $\{(1,0),(0,1)\}$. Thus $X_{\Ga(\ga_1,\ga_2)}$ is non-singular.
\begin{itemize}
\item Consider $\ga_2+\ga_3\in \vn$. Then 
\end{itemize}
\begin{align}
A(\ga_2)&=\{(0,-1),(2,6)\},\notag\\
A(\ga_3)&=\{(0,-2),(2,5)\}.\notag 
\end{align}
The minimal generating set of $\N\Ga(\ga_2,\ga_3)$ is $\{(0,-1),(1,2),(2,6)\}$. By Theorem \ref{Reb}, iterating Nash blowup resolves the singularities of $X_{\Ga(\ga_2,\ga_3)}$.
\begin{itemize}
\item Consider $\ga_3+\ga_4\in \vn$. Then 
\end{itemize}
\begin{align}
A(\ga_3)&=\{(0,-2),(0,-1)\},\notag\\
A(\ga_4)&=\{(-2,-6),(-2,-7)\}.\notag 
\end{align}
The minimal generating set of $\N\Ga(\ga_3,\ga_4)$ is $\{(-2,-6),(0,-1),(1,2),(3,7)\}$. Now consider $\hat{\Ga}=\{\ga_2,\ga_3,\ga_4\}$. Then 
$$\hat{\Ga}(\ga_3,\ga_4)=\{(-2,-6),(0,-1),(1,2),(3,7)\}.$$
In particular, $X_{\hat{\Ga}(\ga_3,\ga_4)}=X_{\Ga(\ga_3,\ga_4)}$ is an affine chart of the Nash blowup of $X_{\hat{\Ga}}$. By Theorem \ref{Reb}, iterating Nash blowup resolves the singularities of $X_{\hat{\Ga}}$. In particular, the same holds for $X_{\Ga(\ga_3,\ga_4)}$.

\end{exam}

\section{The main theorem}\label{sect main}

\begin{thm}\label{main}
Let $\sd\subset\R^2$ be a strongly convex rational polyhedral cone of dimension $2$ and $\Ga$ be the minimal generating set of $\sd\cap\Z^2$. Let $\Theta=\conv(\sd\cap\Z^2\setminus\{(0,0)\})$. Assume that $\Theta$ has only one or two compact edges. Then the iteration of Nash blowups resolves the singularities of $\xg$.
\end{thm}

The proof is divided into several cases which we discuss in Propositions \ref{one segment}, \ref{two segments-1}, and \ref{two segments-2}. In each statement we use the notation of Theorem \ref{main}. The cases are:

\begin{enumerate}
\item $\Theta$ has one compact edge.
\item $\Theta$ has two compact edges, one of them having only two integral points.
\item $\Theta$ has two compact edges, both having more than two integral points.
\end{enumerate}

\begin{prop}\label{one segment}
Assume that $\Theta$ has only one compact edge. Then the Nash blowup of $\xg$ is non-singular.
\end{prop}
\begin{proof}
Up to a change of coordinates, we can assume that the compact edge of $\Theta$ is the segment $[(1,0),(1,k)]$, for some $k\geq2$. Hence its vertices are $(1,0)$ and $(1,k)$. Moreover, $\Ga=\{\ga_1=(1,0),\ga_2=(1,1),\ldots,\ga_{k}=(1,k-1),\ga_{k+1}=(1,k)\}$. Then, Proposition \ref{vertices} implies
$$\vn=\{\ga_1+\ga_2,\ga_{k}+\ga_{k+1}\}.$$
\begin{itemize}
\item Consider $\ga_1+\ga_2\in V(\Nw)$. Then 
\end{itemize}
\begin{align}
A(\ga_1)&=\{(0,2),\ldots,(0,k)\},\notag\\
A(\ga_2)&=\{(0,1),\ldots,(0,k-1)\}.\notag 
\end{align}
Hence, $\Ga(\ga_1,\ga_2)=\{(1,0),(1,1)\}\cup\{(0,1),(0,2),\ldots,(0,k)\}$. The minimal generating set of $\N\Ga(\ga_1,\ga_2)$ is $\{(1,0),(0,1)\}$. Thus $X_{\Ga(\ga_1,\ga_2)}$ is non-singular.
\begin{itemize}
\item Consider $\ga_{k}+\ga_{k+1}\in\vn$. Then
\end{itemize}
\begin{align}
A(\ga_{k})&=\{(0,-(k-1)),\ldots,(0,-1)\},\notag\\
A(\ga_{k+1})&=\{(0,-k),\ldots,(0,-2)\}.\notag 
\end{align}
Hence, $\Ga(\ga_k,\ga_{k+1})=\{(1,k-1),(1,k)\}\cup\{(0,-1),(0,-2),\ldots,(0,-k)\}$. The minimal generating set of $\N\Ga(\ga_k,\ga_{k+1})$ is $\{(1,k),(0,-1)\}$. Thus $X_{\Ga(\ga_k,\ga_{k+1})}$ is non-singular.
\end{proof}

\begin{prop}\label{two segments-1}
Assume that $\Theta$ has two compact edges, one of them having only two integral points. Then iterating the Nash blowup resolves the singularities of $\xg$.
\end{prop}
\begin{proof}
Up to a change of coordinates, we can assume that the compact edges of $\Theta$ are the segments $[(1,0),(1,k)]$, for some $k\geq1$, and $[(1,k),(p,q)]$, where $1<p<\frac{q}{k}$ and $(p,q)=1$. Hence the vertices of $\Theta$ are $(1,0)$, $(1,k)$, and $(p,q)$. Recall that 
$$\Ga=\Big([(1,0),(1,k)]\cup[(1,k),(p,q)]\Big)\cap\Z^2.$$
By Theorem \ref{Reb}, we assume that $|\Ga|\geq4$.
\\
\\
\noindent \textbf{Case 1:} Suppose that $[(1,k),(p,q)]\cap\Z^2=\{(1,k),(p,q)\}$. Write $(p,q)=(1,k)+(a,b)$, where $(a,b)\in\N^2$ and $a,b\geq1$. Since $|\Ga|\geq4$ we have, for some $k\geq2$,
$$\Ga=\{\ga_1=(1,0),\ga_2=(1,1),\ldots,\ga_{k+1}=(1,k),\ga_{k+2}=(1+a,k+b)\}.$$
By Proposition \ref{vertices}, 
$$\vn=\{\ga_1+\ga_2,\ga_{k}+\ga_{k+1},\ga_{k+1}+\ga_{k+2}\}.$$
\begin{itemize}
\item Consider $\ga_1+\ga_2\in V(\Nw)$. Then 
\end{itemize}
\begin{align}
A(\ga_1)&=\{(0,2),\ldots,(0,k),(a,k+b)\},\notag\\
A(\ga_2)&=\{(0,1),\ldots,(0,k-1),(a,k+b-1)\}.\notag 
\end{align}
Hence, $\Ga(\ga_1,\ga_2)=\{(1,0),(1,1)\}\cup\{(0,1),(0,2),\ldots,(0,k),(a,k+b-1),(a,k+b)\}$. The minimal generating set of $\N\Ga(\ga_1,\ga_2)$ is $\{(1,0),(0,1)\}$. Thus $X_{\Ga(\ga_1,\ga_2)}$ is non-singular.
\begin{itemize}
\item Consider $\ga_{k}+\ga_{k+1}\in\vn$. Then
\end{itemize}
\begin{align}
A(\ga_{k})&=\{(0,-(k-1)),\ldots,(0,-1),(a,b+1)\},\notag\\
A(\ga_{k+1})&=\{(0,-k),\ldots,(0,-2),(a,b)\}.\notag 
\end{align}
Hence, $\Ga(\ga_k,\ga_{k+1})$ is the set
$$\{(1,k-1),(1,k)\}\cup\{(0,-1),(0,-2),\ldots,(0,-k),(a,b),(a,b+1)\}.$$ 
A generating set of $\N\Ga(\ga_k,\ga_{k+1})$ is $\{(0,-1),(1,k),(a,b+1)\}$. By Theorem \ref{Reb}, iterating Nash blowup resolves the singularities of $X_{\Ga(\ga_k,\ga_{k+1})}$.
\begin{itemize}
\item Consider $\ga_{k+1}+\ga_{k+2}\in \vn$. Then 
\end{itemize}
\begin{align}
A(\ga_{k+1})&=\{(0,-k),\ldots,(0,-2),(0,-1)\},\notag\\
A(\ga_{k+2})&=\{\ga_1-\ga_{k+2},\ga_2-\ga_{k+2},\ldots,\ga_{k}-\ga_{k+2}\}\notag\\
&=\{(-a,-k-b),(-a,-k+1-b),\ldots,(-a,-1-b)\}.\notag 
\end{align}
A generating set of $\N\Ga(\ga_{k+1},\ga_{k+2})$ is 
$$\{(0,-1),(-a,-1-b),(1,k),(1+a,k+b)\}.$$
Now consider $\hat{\Ga}=\{\ga_{k},\ga_{k+1},\ga_{k+2}\}$. Then 
$$\hat{\Ga}(\ga_{k+1},\ga_{k+2})=\{(0,-1),(-a,-1-b),(1,k),(1+a,k+b)\}.$$
In other words, $X_{\hat{\Ga}(\ga_{k+1},\ga_{k+2})}= X_{\Ga(\ga_{k+1},\ga_{k+2})}$ is an affine chart of the Nash blowup of $X_{\hat{\Ga}}$. By Theorem \ref{Reb}, iterating Nash blowup resolves the singularities of $X_{\hat{\Ga}}$. Hence, the same holds for $X_{\Ga(\ga_{k+1},\ga_{k+2})}$.
\\
\\
\noindent \textbf{Case 2:} Suppose that $[(1,0),(1,k)]\cap\Z^2=\{(1,0),(1,k)\}$. Hence, we must have $k=1$. Since $|\Ga|\geq4$ we have that, for some $m\geq2$, 
$$[(1,1),(p,q)]\cap\Z^2=\{(1,1),(1+a,1+b),\ldots,(1+ma,1+mb)=(p,q)\},$$
where $(a,b)\in\N^2$ and $a,b\geq1$. Hence,
\begin{align}
\Ga=&\{\ga_1=(1,0),\ga_2=(1,1)\}\notag\\
&\cup\{\ga_3=(1+a,1+b),\ldots,\ga_{m+2}=(1+ma,1+mb)\}.\notag
\end{align}
By Proposition \ref{vertices}, $\vn=\{\ga_1+\ga_2,\ga_{2}+\ga_{3},\ga_{m+1}+\ga_{m+2}\}.$
\begin{itemize}
\item Consider $\ga_1+\ga_2\in V(\Nw)$. Then 
\end{itemize}
\begin{align}
A(\ga_1)&=\{(a,b+1),\ldots,(ma,1+mb)\},\notag\\
A(\ga_2)&=\{(a,b),\ldots,(ma,mb)\}.\notag 
\end{align}
Hence, 
\begin{align}
\Ga(\ga_1,\ga_2)=&\{(1,0),(1,1)\}\notag\\
&\cup\{(a,b),\ldots,(ma,mb)\}\notag\\
&\cup\{(a,1+b),\ldots,(ma,1+mb)\}.\notag 
\end{align}
A generating set of $\N\Ga(\ga_1,\ga_2)$ is $\{(1,0),(1,1),(a,b),(a,1+b)\}$. Now consider $\hat{\Ga}=\{\ga_1,\ga_2,\ga_3\}$. Then 
$$\hat{\Ga}(\ga_1,\ga_2)=\{(1,0),(1,1),(a,b),(a,1+b)\}.$$
Hence, $X_{\hat{\Ga}(\ga_{1},\ga_{2})}=X_{\Ga(\ga_{1},\ga_{2})}$ is an affine chart of the Nash blowup of $X_{\hat{\Ga}}$. By Theorem \ref{Reb}, iterating Nash blowup resolves the singularities of $X_{\hat{\Ga}}$. In particular, the same holds for $X_{\Ga(\ga_{1},\ga_{2})}$.
\begin{itemize}
\item Consider $\ga_2+\ga_3\in V(\Nw)$. Then 
\end{itemize}
\begin{align}
A(\ga_2)&=\{(0,-1),(2a,2b),\ldots,(ma,mb)\},\notag\\
A(\ga_3)&=\{(-a,-1-b),(a,b),\ldots,((m-1)a,(m-1)b)\}.\notag 
\end{align}
Hence, 
\begin{align}
\Ga(\ga_2,\ga_3)=&\{(1,1),(1+a,1+b)\}\notag\\
&\cup\{(0,-1),(-a,-1-b)\}\notag\\
&\cup\{(a,b),\ldots,(ma,mb)\}.\notag 
\end{align}
A generating set of $\N\Ga(\ga_2,\ga_3)$ is $\{(1,1),(a,b),(-a,-1-b)\}$. By Theorem \ref{Reb}, iterating Nash blowup resolves the singularities of $X_{\Ga(\ga_2,\ga_3)}$.
\begin{itemize}
\item Consider $\ga_{m+1}+\ga_{m+2}\in V(\Nw)$. Then 
\end{itemize}
\begin{align}
A(\ga_{m+1})&=\{(-(m-1)a,-1-(m-1)b),(-(m-1)a,-(m-1)b),\ldots,(-a,-b)\},\notag\\
A(\ga_{m+2})&=\{(-ma,-1-mb),(-ma,-mb),\ldots,(-2a,-2b)\}.\notag 
\end{align}
Hence, 
\begin{align}
\Ga(\ga_{m+1},\ga_{m+2})=&\{(1+(m-1)a,1+(m-1)b),(1+ma,1+mb)\}\notag\\
&\cup\{(-(m-1)a,-1-(m-1)b),(-ma,-1-mb)\}\notag\\
&\cup\{(-a,-b),\ldots,(-ma,-mb)\}.\notag 
\end{align}
A generating set of $\N\Ga(\ga_{m+1},\ga_{m+2})$ is 
$$\{(1+ma,1+mb),(-a,-b),(-(m-1)a,-1-(m-1)b)\}.$$
By Theorem \ref{Reb}, iterating Nash blowup resolves the singularities of the toric variety $X_{\Ga(\ga_{m+1},\ga_{m+2})}$.
\end{proof}

\begin{prop}\label{two segments-2}
Assume that $\Theta$ has two compact edges, both of them having more than two integral points. Then iterating the Nash blowup resolves the singularities of $\xg$.
\end{prop}
\begin{proof}
Up to a change of coordinates, we can assume that the compact edges of $\Theta$ are the segments $[(1,0),(1,k)]$, for some $k\geq2$, and $[(1,k),(p,q)]$, where $1<p<\frac{q}{k}$ and $(p,q)=1$. Hence the vertices of $\Theta$ are $(1,0)$, $(1,k)$, and $(p,q)$. Then, for some $m\geq2$ and some $(a,b)\in\N^2$, where $a,b\geq1$, $(a,b)=1$, we have 
\begin{align}
\Ga=&\big([(1,0),(1,k)]\cup[(1,k),(p,q)]\big)\cap\Z^2\notag\\
=&\{\ga_1=(1,0),\ldots,\ga_{k+1}=(1,k)\}\notag\\
&\cup\{\ga_{k+2}=(1+a,k+b),\ldots,\ga_{k+m+1}=(1+ma,k+mb)=(p,q)\}.\notag
\end{align}
We will show later in the paper that $b-ka=1$. By Proposition \ref{vertices}, 
$$\vn=\{\ga_1+\ga_2,\ga_{k}+\ga_{k+1},\ga_{k+1}+\ga_{k+2},\ga_{k+m}+\ga_{k+m+1}\}.$$
\begin{itemize}
\item Consider $\ga_1+\ga_2\in V(\Nw)$. 
\end{itemize}
Exactly as in the first item of the proof of Proposition \ref{one segment}, we obtain that the minimal generating set of $\N\Ga(\ga_1,\ga_2)$ is $\{(1,0),(0,1)\}$. Thus $X_{\Ga(\ga_1,\ga_2)}$ is non-singular.
\begin{itemize}
\item Consider $\ga_k+\ga_{k+1}\in V(\Nw)$. Then
\end{itemize}
\begin{align}
A(\ga_{k})&=\{(0,-k+1),\ldots,(0,-1),(a,1+b),\ldots,(ma,1+mb)\},\notag\\
A(\ga_{k+1})&=\{(0,-k),\ldots,(0,-2),(a,b),\ldots,(ma,mb)\}.\notag 
\end{align}
Hence, 
\begin{align}
\Ga(\ga_k,\ga_{k+1})=&\{(1,k-1),(1,k)\}\notag\\
&\cup\{(0,-1),\ldots,(0,-k)\}\notag\\
&\cup\{(a,b),\ldots,(ma,mb)\},\notag\\
&\cup\{(a,1+b),\ldots,(ma,1+mb)\}.\notag
\end{align}
A generating set of $\N\Ga(\ga_k,\ga_{k+1})$ is $\{(1,k),(0,-1),(a,1+b)\}$. By Theorem \ref{Reb}, iterating Nash blowup resolves the singularities of $X_{\Ga(\ga_k,\ga_{k+1})}$.
\begin{itemize}
\item Consider $\ga_{k+1}+\ga_{k+2}\in V(\Nw)$. Then
\end{itemize}
\begin{align}
A(\ga_{k+1})&=\{(0,-k),\ldots,(0,-1),(2a,2b),\ldots,(ma,mb)\},\notag\\
A(\ga_{k+2})&=\{(-a,-k-b),\ldots,(-a,-1-b),(a,b),\ldots,((m-1)a,(m-1)b)\}.\notag 
\end{align}
Hence, 
\begin{align}
\Ga(\ga_{k+1},\ga_{k+2})=&\{(1,k),(1+a,k+b)\}\notag\\
&\cup\{(0,-1),\ldots,(0,-k)\}\notag\\
&\cup\{(a,b),\ldots,(ma,mb)\},\notag\\
&\cup\{(-a,-1-b),\ldots,(-a,-k-b)\}.\notag
\end{align}
First notice that a generating set of $\N\Ga(\ga_{k+1},\ga_{k+2})$ is
$$\{(1,k),(1+a,k+b),(0,-1),(a,b),(-a,-1-b)\}.$$
Now notice that this generating set can be further reduced to the subset $\{(1,k),(a,b),(-a,-1-b)\}.$
By Theorem \ref{Reb}, iterating Nash blowup resolves the singularities of $X_{\Ga(\ga_{k+1},\ga_{k+2})}$.
\begin{itemize}
\item Consider $\ga_{k+m}+\ga_{k+m+1}\in V(\Nw)$. Then
\end{itemize}
\begin{align}
A(\ga_{k+m})&=\{(-(m-1)a,-k-(m-1)b),\ldots,(-(m-1)a,-1-(m-1)b)\}\notag\\
&\cup \{(-(m-1)a,-(m-1)b),(-(m-2)a,-(m-2)b),\ldots,(-a,-b)\}\notag\\
A(\ga_{k+m+1})&=\{(-ma,-k-mb),\ldots,(-ma,-1-mb)\}\notag\\
&\cup \{(-ma,-mb),(-(m-1)a,-(m-1)b),\ldots,(-2a,-2b)\}.\notag
\end{align}
Hence, 
\begin{align}
\Ga(\ga_{k+m},&\ga_{k+m+1})=\notag\\
&\{(1+(m-1)a,k+(m-1)b),(1+ma,k+mb)\}\notag\\
&\cup\{(-a,-b),\ldots,(-ma,-mb)\}\notag\\
&\cup\{(-(m-1)a,-1-(m-1)b),\ldots,(-(m-1)a,-k-(m-1)b)\},\notag\\
&\cup\{(-ma,-1-mb),\ldots,(-ma,-k-mb)\}.\notag
\end{align}
First notice that a generating set of $\Ga(\ga_{k+m},\ga_{k+m+1})$ is
\begin{align}
&\{(1+ma,k+mb),(-a,-b)\}\notag\\
&\cup\{(-(m-1)a,-1-(m-1)b),\ldots,(-(m-1)a,-k-(m-1)b)\}.\notag
\end{align}
Notice that the vectors $(-(m-1)a,-l-(m-1)b)$, for $l\in\{1,\ldots,k\}$, are contained in the cone generated by $\{(1+ma,k+mb),(-a,-b)\}$. On the other hand, it is known (and follows from the results of Section \ref{sect cont frac}) that the determinant of two consecutive vectors in $\Ga$ is $1$. In particular, $\det(\ga_{k+m}\,\ga_{k+m+1})=b-ka=1$. A straightforward computation shows
$$\det((1+ma,k+mb)\,(-a,-b))=\det(\ga_{k+m}\,\ga_{k+m+1})=1.$$
This implies that $\N\Ga(\ga_{k+m},\ga_{k+m+1})$ is generated by the set $\{(1+ma,k+mb),(-a,-b)\}$, i.e., $X_{\Ga(\ga_{k+m},\ga_{k+m+1})}$ is non-singular.
\end{proof}

\begin{rem}
A careful analysis of the proof of Theorem \ref{main} shows that the exact same strategy can be applied to not necessarily saturated semigroups whose minimal generating sets are contained in one or two segments and satisfying some extra conditions. As an example, the same proof of Proposition \ref{one segment} applies to the non-saturated semigroup generated by $\Ga=\{(1,0),(1,3),(1,4),(1,6),(1,7),(1,9)\}$. However, it is not the case for $\Ga=\{(1,0),(1,1),(1,3),(1,4)\}$.

We do not pursue a proof of these other cases since it requires too many restrictions making the result only slightly more general. 
\end{rem}


\section{An example with an arbitrary number of segments.}

It is natural to ask whether the techniques presented in the previous sections extend to more compact edges on $\Theta$. In the following example with three compact edges, an iteration of Nash blowups resolves the singularities of the variety but because of different reasons from the ones we have used in our proofs. This shows that our techniques cannot be applied directly to more general cases.

\begin{exam}\label{fib 4}
Let $\Ga=\{\ga_1=(1,0),\ga_2=(1,1),\ga_3=(2,3),\ga_4=(5,8)\}$. We show that iterating twice the Nash blowup resolves the singularities of $\xg$. By Proposition \ref{vertices},
$$\vn=\{\ga_1+\ga_2,\ga_2+\ga_3,\ga_3+\ga_4\}.$$
Direct computations show:
\begin{align}
\N\Ga(\ga_1,\ga_2)&\mbox{ has as generating set }\{(1,0),(1,1),(1,2),(1,3)\},\notag\\
\N\Ga(\ga_2,\ga_3)&\mbox{ has as generating set }\{(-1,-3),(0,-1),(1,1),(2,3),(3,5),(4,7)\},\notag\\
\N\Ga(\ga_3,\ga_4)&\mbox{ has as generating set }\{(-4,-7),(-1,-2),(2,3),(5,8)\}.\notag
\end{align}
These generating sets correspond, respectively, to
\begin{align}
&[(1,0),(1,3)]\cap\Z^2,\notag\\
&[(-1,-3),(4,7)]\cap\Z^2,\notag\\
&[(-4,-7),(5,8)]\cap\Z^2.\notag
\end{align}
By Proposition \ref{one segment}, the Nash blowup of $X_{\Ga(\ga_1,\ga_2)}$, $X_{\Ga(\ga_2,\ga_3)}$, and $X_{\Ga(\ga_3,\ga_4)}$ is non-singular. Notice that the generating set of $\N(\ga_2,\ga_3)$ is neither of cardinality two or three and it cannot be obtained after applying once the algorithm to a semigroup generated by three elements.
\end{exam}

As an invitation to further investigate Nash blowups of normal toric surfaces, now we generalize what we did in the previous example for an arbitrary number of compact edges.

\begin{thm}\label{Fib}
Consider the Fibonacci sequence: $f_1:=1$, $f_2:=1$, and $f_j:=f_{j-1}+f_{j-2}$, for $j>2$. Let $l\geq 4$ be an even integer. Denote 
$$\Ga=\{(1,0),(f_1,f_2),(f_3,f_4),\ldots,(f_{l-1},f_{l})\}\subset\Z^2.$$
Then two iterations of Nash blowup resolve the singularities of $\xg$.
\end{thm}

We need the following lemma for the proof of the theorem.

\begin{lem}\label{l fib}
With the notation of Theorem \ref{Fib}, assume that $l\geq6$. Fix an even number $i\in\{2,\ldots,l-4\}$. Let $f_{-1}=1, f_0=0$. Consider the following set of vectors:
\begin{align}
v_1&=(f_{i-3}-f_{i+1},f_{i-2}-f_{i+2}),\notag\\
v_2&=(f_{i-3}-f_{i-1},f_{i-2}-f_{i}),\notag\\
v_3&=(f_{i-1},f_{i}),\notag\\
v_4&=(f_{i+1},f_{i+2}),\notag\\
v_5&=(f_{i+3}-f_{i+1},f_{i+4}-f_{i+2}),\notag\\
v_6&=(f_{i+3}-f_{i-1},f_{i+4}-f_{i}).\notag
\end{align}
The vectors $v_1,\ldots,v_6$ satisfy the following properties:
\begin{enumerate}
\item $\det(v_j\,v_{j+1})=1$, for each $1\leq j \leq 5$. In particular, these vectors are ordered increasingly counterclockwise.
\item Let $L(x,y)=f_{i+1}x-f_iy$. Then $L(v_j)=1$ for all $1\leq j \leq 6$.
\item $[v_1,v_6]\cap\Z^2=\{v_1,\ldots,v_6\}$.
\end{enumerate}
\end{lem}
\begin{proof}
First notice that item 3 follows from 1 and 2. Now recall Vajda's identity, which is key to our arguments \cite{V}. For all $n,r,s\geq1$ the following holds
\begin{equation}\label{Vaj}
f_nf_{n+r+s}-f_{n+r}f_{n+s}=(-1)^{n+1}f_rf_s.   \tag{*}
\end{equation}
Let us prove item $1$ using (\ref{Vaj}). Recall that $i$ is an even number. If $i=2$ then $v_1=(1,0)$ and $v_2=(1,1)$ and $\det(v_1\,v_2)=1$. Assume $i\geq4$. Then
\begin{align*}
\det(v_1\,v_2)=&f_{i-3}f_{i-2}-f_{i-3}f_{i}-f_{i-2}f_{i+1}+f_if_{i+1}\\
&-\big(f_{i-3}f_{i-2}-f_{i-3}f_{i+2}-f_{i-2}f_{i-1}+f_{i-1}f_{i+2}\big)\\
=&-(f_{i-3}f_{i}-f_{i-2}f_{i-1})+(f_{i-3}f_{i+2}-f_{i-2}f_{i+1})-(f_{i-1}f_{i+2}-f_{i}f_{i+1})\\
=&(-1)^{i-1}f_1f_2+(-1)^{i-2}f_1f_4+(-1)^{i+1}f_1f_2\\
=&-1+3-1=1.
\end{align*}
The same elementary computations show that $\det(v_j\,v_{j+1})=1$ for the remaining cases.

Now we prove item 2. This follows as in 1, let us illustrate the computation for $v_6$.
\begin{align*}
L(v_6)&=L((f_{i+3}-f_{i-1},f_{i+4}-f_i))=f_{i+1}(f_{i+3}-f_{i-1})-f_{i}(f_{i+4}-f_{i})\\
&=-(f_{i}f_{i+4}-f_{i+1}f_{i+3})-(f_{i-1}f_{i+1}-f_if_i)\\
&=(-1)^{i+2}f_1f_3+(-1)^{i+1}f_1f_1\\
&=2-1=1.
\end{align*}
\end{proof}

\begin{proof}
(of Theorem \ref{Fib}) Let $\sd=\R_{\geq0}((1,0),(f_{l-1},f_{l}))$. In the next section we show that $\Ga$ is the minimal generating set of the semigroup $\sd\cap\Z^2$ (see Proposition \ref{ct fib}). In particular, $\Ga$ coincides with the set of integral points of the compact edges of $\Theta=\conv(\sd\cap\Z^2\setminus\{(0,0)\})$.

Notice that the compact edges of $\Theta$ are the segments $[(1,0),(f_1,f_2)]$ and $[(f_{i-1},f_{i}),(f_{i+1},f_{i+2})]$ for each even number $2\leq i\leq l-2$. Indeed, a direct application of Vajda's identity (\ref{Vaj}) shows that the slopes of these segments are strictly decreasing:
$$\frac{f_i-f_{i-2}}{f_{i-1}-f_{i-3}}=\frac{f_{i-1}}{f_{i-2}}>\frac{f_{i+1}}{f_i}=\frac{f_{i+2}-f_{i}}{f_{i+1}-f_{i-1}}.$$
By Proposition \ref{vertices},
$$\vn=\{(1,0)+(f_1,f_2),(f_1,f_2)+(f_3,f_4),\ldots,(f_{l-3},f_{l-2})+(f_{l-1},f_{l})\}.$$
The case $l=4$ can be verified by a direct computation. Assume $l\geq6$.
\begin{itemize}
\item Consider $(1,0)+(f_1,f_2)\in\vn$. Then
\end{itemize}
\begin{align}
A((1,0))&=\{(1,3),(f_{i-1}-1,f_{i})|6\leq i \leq l, i\mbox{ even}\},\notag\\
A((f_1,f_2))&=\{(1,2),(f_{i-1}-1,f_{i}-1)|6\leq i \leq l, i\mbox{ even}\}.\notag 
\end{align}
Hence, 
\begin{align}
\Ga((1,0),(f_1,f_2))=&\{(1,0),(1,1),(1,2),(1,3)\}\notag\\
&\cup\{(f_{i-1}-1,f_{i}),(f_{i-1}-1,f_{i}-1)|6\leq i \leq l, i\mbox{ even}\}.\notag
\end{align}
We claim that $\N\Ga((1,0),(f_1,f_2))$ can be minimally generated by the set $\{(1,0),(1,1),(1,2),(1,3)\}$. Consider first the following general fact, which can be proved by a straightforward induction: for $k\geq5$, any two consecutive Fibonacci numbers satisfy $f_k\leq2(f_{k-1}-1)$. In particular, for any $i\in\{6,\ldots,l\}$, $i$ even,
$$1<\frac{f_{i}-1}{f_{i-1}-1}<\frac{f_{i}}{f_{i-1}-1}\leq2.$$
Hence, $(f_{i-1}-1,f_{i}),(f_{i-1}-1,f_{i}-1)\in\N((1,1),(1,2))$, for any $i\in\{6,\ldots,l\}$, $i$ even. This proves the claim. Finally, by Proposition \ref{one segment}, the Nash blowup of $X_{\Ga((1,0),(f_1,f_2))}$ is non-singular.

\begin{itemize}
\item Consider $(f_{l-3},f_{l-2})+(f_{l-1},f_{l})\in\vn$.
\end{itemize}
 This case is actually symmetric to the previous one, in the following sense. We claim that there exists a linear automorphism $T$ of $\Z^2$ such that 
\begin{align}
T(f_{l-1},f_{l})=&(1,0),\notag\\
T(f_{l-3},f_{l-2})=&(1,1)=(f_1,f_2),\notag\\
T(f_{l-5},f_{l-4})=&(2,3)=(f_3,f_4),\notag\\
&\vdots\notag\\
T(f_1,f_2)=&(f_{l-3},f_{l-2}),\notag\\
T(1,0)=&(f_{l-1},f_l).\notag
\end{align}
In particular, the conclusion of the previous case also holds for this case. Define
\begin{align}
T:=
\begin{pmatrix}
f_{l-1}&	-f_{l-2}\\
f_{l}&		-f_{l-1}\notag
\end{pmatrix}.
\end{align}
Firstly, $T(1,0)=(f_{l-1},f_l)$. 
Letting $n=l-2$ and $r=s=1$, Vajda's identity (\ref{Vaj}) implies $f_{l-2}f_l-f_{l-1}^2=(-1)^{l-1}f_1f_1=-1$. Hence, $\det T=-1$ and $T(f_{l-1},f_l)=(1,0)$. Now let $k$ be an even integer, $4\leq k<l$. Using Vajda's identity we obtain $T(f_{k-1},f_{k})=(f_{l-k-1},f_{l-k})$.
\begin{itemize}
\item Consider $(f_{i-1},f_{i})+(f_{i+1},f_{i+2})\in\vn$, where $i\in\{2,\ldots,l-4\}$ is even. As in Lemma \ref{l fib}, denote $f_{-1}=1$, $f_0=0$.
\end{itemize}
Consider the following sets:
\begin{align*}
A_{i-1,i}^-&=\{(f_{k-1}-f_{i-1},f_k-f_i)\mid k\in\{0,\ldots,i-2\}\},\\
A_{i-1,i}^+&=\{(f_{k-1}-f_{i-1},f_k-f_i)\mid k\in\{i+4,\ldots,l\}\},\\
A_{i+1,i+2}^-&=\{(f_{k-1}-f_{i+1},f_k-f_{i+2})\mid k\in\{0,\ldots,i-2\}\}\\
A_{i+1,i+2}^+&=\{(f_{k-1}-f_{i+1},f_k-f_{i+2})\mid k\in\{i+4,\ldots,l\}\}.
\end{align*}
Then $A((f_{i-1},f_i))=A_{i-1,i}^-\cup A_{i-1,i}^+$ and $A((f_{i+1},f_{i+2}))=A_{i+1,i+2}^-\cup A_{i+1,i+2}^+$. Hence,
\begin{align*}
\Ga((f_{i-1},f_{i}),(f_{i+1},f_{i+2}))=&\{(f_{i-1},f_i),(f_{i+1},f_{i+2})\}\\
&\cup A((f_{i-1},f_i)) \cup  A((f_{i+1},f_{i+2})).
\end{align*}
Notice that the vectors $v_1,\ldots,v_6$ from Lemma \ref{l fib} appear in the above set. Indeed, $v_1\in A_{i+1,i+2}^-$, $v_2\in A_{i-1,i}^-$, $v_3=(f_{i-1},f_i)$, $v_4=(f_{i+1},f_{i+2})$, $v_5\in A_{i+1,i+2}^+$, and $v_6\in A_{i-1,i}^+$. We claim that 
\begin{equation}\label{v1v6}
\Ga((f_{i-1},f_{i}),(f_{i+1},f_{i+2}))\subset\R_{\geq0}(v_1,v_2,v_3,v_4,v_5,v_6).
\end{equation}
Firstly, since the sequence $\{f_j\}_{j\geq2}$ is strictly increasing, we have
\begin{align}
A_{i-1,i}^- \cup A_{i+1,i+2}^-&\subset \Z_{\leq0}\times\Z_{<0},\label{A-}\\
A_{i-1,i}^+ \cup A_{i+1,i+2}^+&\subset \Z_{>0}\times\Z_{>0}.\label{A+}
\end{align}
In particular, $v_2\in\Z_{\leq0}\times\Z_{<0}$. Moreover, $v_3\in\Z_{>0}\times\Z_{>0}$. By Lemma \ref{l fib} we know that $\det(v_2\,v_3)=1$. This implies that $(0,-1),(1,0)\in\R_{\geq0}(v_2,v_3)\subset\R_{\geq0}(v_1,\ldots,v_6)$.

Recall that $\Ga$ is the set of integral points in the compact edges of $\Theta=\conv(\sd\cap\Z^2\setminus\{(0,0)\})$, where $\sd=\R_{\geq0}((1,0),(f_{l-1},f_l))$. Putting these facts together gives:

\begin{itemize}
\item By the convexity of $\Theta$ and (\ref{A-}) we have $A_{i+1,i+2}^-\subset\R_{\geq0}(v_1,(0,-1)).$
\item By the convexity of $\Theta$ and (\ref{A+}) we have $A_{i-1,i}^+\subset\R_{\geq0}((1,0),v_6).$
\item By the convexity of $\Theta$ and (\ref{A-}) we have $A_{i-1,i}^-\subset\R_{\geq0}(v_2,(0,-1)).$
\item By the convexity of $\Theta$ and (\ref{A+}) we have $A_{i+1,i+2}^+\subset\R_{\geq0}((1,0),v_5).$
\end{itemize}
This proves (\ref{v1v6}).

Let $S=\Ga((f_{i-1},f_{i}),(f_{i+1},f_{i+2}))\setminus\{v_1,\ldots,v_6\}$. Let $L(x,y)=f_{i+1}x-f_iy$. By (\ref{v1v6}) and 2 of Lemma \ref{l fib}, it follows that $L(w)>1$ for all $w\in S$. By 1 of the same lemma, we have $S\subset\Z_{\geq0}(v_1,\ldots,v_6)$. Hence,  $\{v_1,\ldots,v_6\}$ is the minimal generating set of the semigroup generated by $\Ga((f_{i-1},f_{i}),(f_{i+1},f_{i+2}))$. Finally, by 3 of Lemma \ref{l fib} and Proposition \ref{one segment} we conclude that the Nash blowup of $X_{\Ga((f_{i-1},f_{i}),(f_{i+1},f_{i+2}))}$ is non-singular.
\end{proof}


\section{Normal toric surfaces and continued fractions}\label{sect cont frac}

It is known that the minimal generating set of the semigroup defining a normal toric surface can be described in terms of a continued fraction. In this final section we describe the continued fraction of the semigroups we studied in previous sections. We also deduce some facts on the dual
graph of the minimal resolution of the normal toric surfaces in question.

Consider a strongly convex rational polyhedral cone $\sd\subset \R^2$ in normal form, that is, it is generated by $(1,0)$ and $(p,q)$ with $p$ and $q$ relatively prime integers such that $0\leq p<q$. The minimal set of generators of $\sd\cap \Z^2$ can be described in terms of the continued fraction associated to $\frac{p}{q}$ as follows. 

Denote by $[a_1, \ldots , a_n]$ the continued fraction:
$$\frac{p}{q} = a_1 - \frac{1}{a_2 - \frac{1}{... - \frac{1}{a_n}}}.$$
We call it the continued fraction associated to $\sd\cap \Z^2$. The number $n$ is called the length of the continued fraction. Now define the integers $p_i$ and $q_i$ as follows:
\begin{align*}
p_{-1} &= 0, \,\,\,\,\,  p_0 = 1, \,\,\,\,\, p_i = a_ip_{i-1} - p_{i-2}, \, \mbox{ for } i\geq1,\\ 
q_0 &= 0, \,\,\,\,\, q_1=1, \,\,\,\,\, q_i = a_iq_{i-1} - q_{i-2}, \, \mbox{ for } i\geq2.
\end{align*}
The minimal set of generators of the semigroup $\sd\cap \Z^2$ is precisely the set of pairs $(p_i, q_i)$ for $0\leq i \leq n$ \cite[Proposition 4.9 and Theorem 4.10]{At}.

When the generators of the semigroup defining the toric surface lie in one or two segments, the continued fraction is of a special type.

\begin{prop}\label{fraction-Gamma}
Let $\sd=\R_{\geq 0}((1,0),(p,q))\subset\R^2$ be a  strongly convex cone in its normal form. Let $\Ga\subset\Z^2$ be the minimal generating set of $\sd\cap\Z^2$. 
\begin{enumerate}
\item If $\Ga$ is contained in a single segment, then the continued fraction of $\frac{p}{q}$ is of the form $[1,2, \ldots, 2]$.
\item If $\Ga$ lies in two segments then the continued fraction associated to $\frac{p}{q}$ is of the form $[1,2,\ldots,2, x,2, \ldots, 2]$. 
\end{enumerate}
In both cases, the length of the continued fractions is one less than the cardinality of $\Gamma$. 

\end{prop}
\begin{proof}

Suppose that $\Ga$ is contained in two segments. Then $\Ga$ has the form
$$\{(1,0), (1,1), \ldots , (1,k), (1+a,k+b), (1+2a,k+2b),\ldots,(1+ra,k+rb)=(p,q)\},$$
for some $a,b,k\in\N_{\geq1}$ such that $(a,b)=(b,k)=1$ and $b-ak=1$. Hence, $b>ak$.

Let $\frac{p}{q}=[a_1,\ldots,a_n]$. From the definition of the $p_i$'s and $q_i$'s given above, it follows that:
$$a_1 = \frac{p_1}{p_0}, \, \,  {\rm and} \, \,  a_i = \frac{p_i + p_{i-2}}{p_{i-1}}, \, \,  {\rm for}\, \,  i\geq 2.$$ 
Moreover, since $(p_i, q_i)$ are the coordinates of the elements of $\Gamma$, we have: $p_i =1$ for $0\leq i \leq k$. 
Therefore: 
$$a_1 =1, \, \,  {\rm and}\, \,  a_i = 2 \, \, {\rm for} \, \,  2\leq i\leq k.$$ 
On the other hand, for $1\leq j \leq r$, $p_{k+j} =1+ja$.  Therefore: 
$$a_{k+1} =\frac{a+1 + p_{k-1}}{p_k} = a+2,\,\,\, a_{k+2} = \frac{2a+2}{a+1} =2,$$
and then
$$a_{k+j} = 2, \, \,  {\rm for}\, \,  2\leq j\leq r.$$
We conclude that the continued fraction associated to $\Gamma$ is 
$$[1,2, \ldots , 2, a+2, 2, \ldots , 2]$$ 
and has length $k+r$. 

When $\Gamma$ is contained in one segment, the process of computing the continued fraction is the same as above, stopping at $i=k$. In which case, 
$\frac{p}{q}= [1,2,\ldots ,2]$ and has length $k$. 
\end{proof}

Let us now compute the inverse of the above continued fraction. 

\begin{prop}\label{inverse-fraction}
Consider the continued fraction $\frac{p}{q} = [1, 2, \ldots 2, x, 2, \ldots , 2]$ such that:
\begin{itemize}
\item $x\geq 2$.
\item $x$ appears in the position $n\geq 2$.
\item The total length of the continued fraction is $m\geq 2$. If $x=2$ then $n=m$.
\end{itemize} 
Then its inverse is the continued fraction $\frac{q}{p} = [n,2, \ldots , 2, m-n+2]$, and the length of this continued fraction is $x-1$. 
 
Conversely, every continued fraction of the form $[\alpha, 2, \ldots 2, \beta]$ of length $\nu \geq1$, $\alpha , \beta \geq 2$, is the inverse of a continued fraction $[1, 2, \ldots , 2, \nu + 1, 2, \ldots, 2]$ of length $\beta$.
\end{prop}
\begin{proof}
The statement follows by a straightforward multi-induction on the considered values.    
\end{proof}

Now we want to describe the set $\Ga$ of Theorem \ref{Fib} in terms of continued fractions. Let $\Gamma = \{(1,0), (f_1, f_2), \ldots , (f_{l-1}, f_l)=(p,q)\}$, with $l\geq 4$ an even integer, be as in that theorem.

\begin{prop}\label{ct fib}
The continued fraction associated to $\Gamma$ is 
$$\frac{p}{q} = [1,3,\ldots , 3]$$
and its length is $l/2$. Moreover, the continued fraction of its inverse $\frac{q}{p}$ is $[2, 3,\ldots 3, 2]$ and has length $l/2$.
\end{prop}
\begin{proof}
    From the description of the continued fraction $\frac{p}{q} = [a_1, \ldots a_n]$, in the beginning of this section, we have:

    $$a_i = \frac{p_i + p_{i-2}}{p_{i-1}} \, \, {\rm and}\, \,  p_i = f_{2i-1}.$$ 

    Consequently, 
    
    $$a_1 = \frac{f_1}{p_0} = 1, a_2= \frac{f_3 + 1}{1} =3.$$

    For $h\geq 2$. We have:
    
    $$a_{h+1} = \frac{p_{h+1} +p_{h-1}}{p_h}=\frac{f_{2h+1}+f_{2h-3}}{f_{2h-1}} 
    = \frac{3(f_{2h-3} + f_{2h-2})}{f_{2h-3}+f_{2h-2}}= 3.$$

The second statement of the proposition can be achieved, again,  by a straightforward induction on the length of the continued fraction. Observe that, when $n=2$, $\frac{p}{q}=[1,3]$, and $\frac{q}{p}= [2,2]$.
\end{proof}

\begin{rem}
Let us consider the dual diagram of the minimal resolution of a normal toric surface. It corresponds to curves $E_i$, $i=1,\ldots,s$, such that each curve $E_i$ is a smooth rational irreducible curve, with $E_i$ intersecting only and transversally $E_{i+1}$ and $E_{i-1}$, in one point each. This configuration is usually called a bamboo.
 
Consider the cone $\sd=\R_{\geq0}((1,0),(p,q))$ determining the toric surface $X_{\Gamma}$. The continued fraction associated to the rational number $\frac{q}{p}$ determines the self-intersections of the $E_i$'s. In other words, if $\frac{q}{p} = [b_1, \ldots , b_s]$, then $E_i^2 = - b_i$ \cite[Section 2.6]{Ful} and \cite[Section III.4]{Ba}. Propositions \ref{inverse-fraction} and \ref{ct fib} give the self-intersections of the divisors in the cases we studied.
\end{rem}

\begin{rem}
Consider a normal toric surface $X$ whose minimal resolution has as exceptional divisor $E= \bigcup_{i=1}^s E_i$, with self-intersections $E_1^2 = -\alpha\leq -2$, $E_s^2 = - \beta\leq -2$ and $E_i^2 = -2$, for $2\leq i<s$. By Proposition \ref{inverse-fraction} and Theorem \ref{main}, $X$ can be resolved by a finite iteration of Nash blowups.
\end{rem}

\vspace{.5cm}
\noindent{\footnotesize \textsc {Daniel Duarte, Centro de Ciencias Matem\'aticas, UNAM Campus Morelia.} \\ E-mail: adduarte@matmor.unam.mx}\\
{\footnotesize \textsc {Jawad Snoussi, Universidad Nacional Aut\'onoma de M\'exico, Instituto de Mate-m\'aticas, Unidad Cuernavaca.} \\ E-mail: jsnoussi@im.unam.mx}

\end{document}